\documentclass[letterpaper,twoside,12pt]{article}

\usepackage[utf8]{inputenc}
\usepackage[english]{babel}
\usepackage[letterpaper, margin=1in]{geometry}

\usepackage{color}
\usepackage{multicol}
\usepackage{amsmath}
\usepackage{amssymb}
\usepackage{enumerate}
\usepackage{array}
\usepackage{longtable}
\usepackage{graphicx}
\usepackage{pst-all}
\usepackage{yfonts}
\newpsobject{grilla}{psgrid}{subgriddiv=1,griddots=10,gridlabels=6pt}

\usepackage[T1]{fontenc}

\usepackage{textcomp}

\usepackage{lmodern} 
\usepackage{amsmath}
\usepackage{amsthm}
\usepackage{amsfonts}

\usepackage{nameref}
\usepackage{varioref}
\usepackage{hyperref}
\hypersetup{ 
    pdfborder = {0 0 0} 
}

\usepackage{amsmath}
\usepackage{amsthm}
\usepackage{amsfonts}

\usepackage[all]{xypic}

\theoremstyle{plain}
\newtheorem{thm}{Theorem}
\newtheorem{lem}{Lemma}[section]

\newtheorem{cor}{Corollary}[section]
\theoremstyle{definition}
\newtheorem{exam}{Example}[section]
\newtheorem{defin}{Definition}[section]
\newtheorem{rem}{Remark}[section]

\begin{document}

\pagestyle{empty}

\begin{center}
\textbf{\large Lifting of linear connections and geodesic simulation}
\end{center}

\begin{center}
Carlos A. Hern\'andez
\end{center}

\pagestyle{plain}

\begin{abstract}
This article describes and classifies the liftings of a connection on a differential manifold $M$. Such lifting gives a connection on the frame bundle $L(M)$ which is compatible with the original connection via the canonical projection. A specific connection, called \textit{induced lifting}, will be used as reference to classify the others. Another kind of liftings, called \textit{adjust liftings}, will be used for geodesic modeling of second order differential equations involving the connection on the base $M$. To conclude, an application of this modeling to boundary value problems will be given to conclude.
\end{abstract}

\section{Introduction}

The study of connection liftings is interesting because some liftings in particular allows to perform geodesic simulation in order to model families of solutions of differential equations on $M$ as projections of geodesics in $L(M)$ (see Theorem \ref{teolevaj}).

This new example of geodesic simulation is important advance in the field, because it allows us to perform geodesic simulation for a wide general kind of equation with a simple form, giving us the possibility to model any second order equation for a dissipative system with a dissipative  force that does not vanish anywhere. Let $\nu(X)=u\omega(X)u^{-1}$, thus called vertical protection. The main theorem of this paper is as follows.

\begin{thm}\label{teolevaj}
Given a differential equation over a manifold $M$ with connection $\widetilde{\nabla}$, of the form
\begin{equation}\label{eqteolevaj}
\widetilde{\nabla}_{\dot{\gamma}}\dot{\gamma}=F\dot{\gamma},
\end{equation}
with $F\in\Omega_{1}^{1}(M)$ always different from zero, there exist a lift $\nabla$ of $\widetilde{\nabla}$ for which geodesics with initial condition $X\in T_{u}(L(M))$ such that $\nu(X)=F(\pi(u))$ correspond, when the canonical projection is applied, with the solution of \eqref{eqteolevaj} with initial condition $\pi_{*}(X)$. Hence, all solutions of \eqref{eqteolevaj} can be obtained from $\nabla$ by means of geodesic simulation.
\end{thm}

This point of view allows to study differential equations from a different perspective, in particular, this work presents an application of Theorem \ref{teolevaj}, namely Theorem \ref{teofront}, concerning boundary value problems, which might have applications on control theory.

\emph{Lifting a map} is a general and important problem in mathematics. Given applications $\pi : X \to Y$ and $f : Z \to Y$ when can we \emph{lift} $f$ to a map $F : Z \to X$ such that the diagram commutes?
\begin{equation}
\xymatrix{
& X \ar[d]^{\pi}
\\
Z \ar[ru]^F  \ar[r]_f & Y 
}
\label{eq:0000}
\end{equation}

A generalization of this concept leads to the lifting of sections of bundles over smooth varieties:
Given a bundle $\pi : E \to M$ and a vector field $V : M \to TM$, to find a field $W : E \to TE$ such that $d\pi \circ W = V \circ \pi$. For a reference, see Ivan Kol{\'a}ř, Jan Slov{\'a}k y Peter W. Michor \cite{KSM}.

There are many works about lifting of connections, see for instance \cite{PSB}, \cite{CDL}, \cite{Gasior}, \cite{Mik}, \cite{Mik1}, \cite{Shapukov}, \cite{SSS}, \cite{Cordero_Leon}, \cite{Mok} and \cite{Salimov_Magden}. Instead of using local coordinates, like the works cited basically do, this work \emph{describe and classify lifts in geometrical terms, making the emphasis on the corresponding parallel transport and its properties}.

The lifting of connections will be applied to geodesic modeling, meaning that families of solutions of differential equations will appear when applying canonical projection to geodesics in the frame bundle, to do this specific kind of liftings will be used. Many results of this kind are known, see \cite{AMT}, \cite{MC}, \cite{IgoShaYak}, \cite{Igoshin}, \cite{SGL}, \cite{Malakhaltsev}, \cite{RM}, \cite{BON}, \cite{Yakovlev}, \cite{Solodovnikov}, \cite{Sinyukov} and \cite{Bejana}.

In section 2 we present background theory. In section 3 we study induced lifting and its properties. Section 4 uses the properties of induced lifting to classify the other liftings. In section 5 we study adjust liftings and geodesic simulation. Finally, section 6 shows examples and an relevant application to boundary value problems.

\section{Horizontality and verticality on the frame bundle}

Keeping in mind the notions and importance of the canonical and connection forms as well as the fundamental and standard horizontal vector fields (see \cite{KN}), let $M$ be a differential manifold with a linear connection and $L(M)$ its frame bundle. Together with the notion of $L(M)$ we have a vertical distribution given by the tangent spaces to the fibers and the connection itself can be given as a horizontal complementary distribution. In this way, every vector in $T(L(M))$ can be uniquely decomposed into horizontal and vertical components.

Related to the horizontal component we have the canonical form $\theta$ which takes values on the generic fiber $F=\mathbb{R}^{n}$. On the other hand, the vertical component can be associated with the connection form $\omega$ which takes values on the Lie algebra of the structural group $\mathfrak{g}=\mathfrak{gl}(n)$.

While the notion of horizontality depends on the connection, the canonical form does not; on the contrary, $\omega$ does depend on the connection (the kernels corresponding to the horizontal distribution) while the notion of verticality and the values of $\omega$ on the vertical distribution does not.

Finally, we have the fundamental fields $A^{*}$, which are vertical vector fields where the value of $\omega$ is fixed $\omega(A^{*})=A\in\mathfrak{g}$; and we have standard horizontal vector fields $B(\xi)$ where the values of $\theta$ is fixed $\theta(B(\xi))=\xi\in F$. As the notion of verticality, the notion of fundamental field does not depend on the connection.

\section{Induced parallel transport}

\begin{defin}\label{deftrpi}
Given a linear connection $\Gamma$ over a manifold $M$, let $\tau:[0,1]\rightarrow L(M)$ be a smooth curve $u=\tau_{0}$, we define the \textit{parallel transport induced} by $\Gamma$ of the vector $X\in T_{u}(L(M))$ along $\tau$ as
\begin{equation}\label{eqtrpi}
\tau(X)=(\widetilde{\tau}(\pi_{*}(X)))^{w}+A^{*}(w),
\end{equation}
where $w=\tau_{1}$, $\omega(X)=A\in\mathfrak{g}$, $(\widetilde{\tau}(\pi_{*}(X)))^{w}\in T_{w}(L(M))$ is the horizontal lift of $\widetilde{\tau}(\pi_{*}(X))\in T_{\pi(w)}(M)$ and $\widetilde{\tau}$ represents parallel transport along $\pi(\tau)$ in $TM$.
\end{defin}

\begin{thm}\label{teotrpi}
The induced parallel transport (or induced transport) is the unique linear connection on $T(L(M))$ such that for every smooth curve $\tau$ on $L(M)$, $u=\tau_{0}$ and $X\in T_{u}(L(M))$ satisfies

\begin{enumerate}
\item The induced transport preserves the value of the connection form $\omega$
\[
\omega(\tau(X))=\omega(X).
\]

\item The induced transport along horizontal paths preserves the value of the canonical form $\theta$. If $\tau$ is horizontal
\[
\theta(\tau(X))=\theta(X).
\]

\item The induced transport along vertical paths preserves the value of the differential of the canonical projection $\pi_{*}$. If $\tau$ is vertical
\[
\pi_{*}(\tau(X))=\pi_{*}(X).
\]

\item The induced transport $\tau$ on $T(L(M))$ commutes with the differential of the canonical projection $\pi_{*}$ and the parallel transport $\widetilde{\tau}=\pi(\tau)$ on $TM$
\[
\pi_{*}(\tau(X))=\widetilde{\tau}(\pi_{*}(X)).
\]
\end{enumerate}
\end{thm}
\begin{proof}
Properties (1), (3) and (4) follow directly from the construction of the induced transport. Property (2) follows, for a horizontal $\tau$ is the horizontal lift of $\widetilde{\tau}=\pi(\tau)$ from $u=\tau_{0}$, so $w=\tau_{1}$ is the parallel transport of $u=\tau_{0}$ along $\widetilde{\tau}$. Given $X\in T_{u}(L(M))$
\[
\theta(X)=u^{-1}(\pi_{*}(X))=(\widetilde{\tau}(u))^{-1}(\widetilde{\tau}(\pi_{*}(X)))=w^{-1}(\pi_{*}(\tau(X)))=\theta(\tau(X)).
\]
Uniqueness follows from linearity of induced transport, given that properties (1) and (4) determine the induced transport of both horizontal an vertical vectors.
\end{proof}

\begin{thm}\label{teogeoi}
Given a linear connection $\Gamma$ on a manifold $M$. Let $u\in L(M)$, $X\in T_{u}(L(M))$, $A\in\mathfrak{g}$ such that $v(X)=A^{*}(u)$ and $a_{t}=\exp(tA)$. The geodesic on $L(M)$, determined by the parallel transport induced by $\Gamma$, which passes through $u$ in the direction of the vector $X$, is given by $\gamma_{t}=u_{t}a_{t}=(ua_{t})_{t}$ where $u_{t}$ with $u_{0}=u$ is the horizontal lift from $u$ of the geodesic on $M$ which passes through $\pi(u)$ in the direction of the vector $\pi_{*}(X)$. Also the canonical projection of $\gamma$, $\pi(\gamma_{t})=\pi(u_{t})$, is a geodesic on $M$ or a constant map.
\end{thm}
\begin{proof}
The equation $u_{t}a_{t}=(ua_{t})_{t}$ means we can lift from $u$ and then apply the action of $a_{t}$, or inversely, act first by $a_{t}$ and then do the lift, both will coincide by virtue of the $G$-invariance of the connection. On the other side, given $\dot{\gamma_{t}}=\dot{u_{t}}a_{t}+u_{t}\dot{a_{t}}=(R_{a_{t}})_{*}(\dot{u}_{t})+A^{*}(\gamma_{t})$, we see $v(\dot{\gamma_{t}})=A^{*}(\gamma_{t})$ and $\pi_{*}(\dot{\gamma_{t}})=\pi_{*}(\dot{u}_{t})$. 

Clearly the field $A^{*}(\gamma_{t})$ is invariant by the induced transport. Also $h(\dot{\gamma_{t}})$, which is the horizontal lift of $\pi_{*}(\dot{u}_{t})$, is invariant by the induced transport along $\gamma$ given that $\pi(u_{t})$ is a geodesic on $M$ so $\pi_{*}(\dot{u}_{t})$ is invariant by the parallel transport along $\pi(\gamma)$ on $TM$, the same is true if $\pi(\gamma)$ is constant.
\end{proof}

\begin{lem}\label{lemdicf}
The induced transport is defined so that the fundamental fields are invariant, and hence
\begin{equation}\label{eqdicf}
\nabla_{X}A^{*}=0,
\end{equation}
for every $A\in\mathfrak{g}$ and every $X\in T(L(M))$.
\end{lem}

\begin{proof}
Is immediate since fundamental fields are parallel.
\end{proof}

\begin{lem}\label{lemdich}
For every $\xi,\zeta\in F$
\begin{equation}\label{eqdich}
\nabla_{B(\xi)}B(\zeta)=0.
\end{equation}
\end{lem}

\begin{proof}
The induced transport preserves horizontality of vectors, and also, along horizontal curves, induced transport preserves the value of the canonical form. Given that $B(\xi)$ is by definition the only horizontal field such that $\theta(B(\xi))=\xi$ the result follows.
\end{proof}

\begin{thm}\label{teodevi}
Let $u\in L(M)$, $\xi\in F$ and $A\in\mathfrak{g}$, then
\begin{equation}\label{eqdevi1}
\nabla_{A^{*}}B(\xi)=B(A\xi),
\end{equation}
meaning
\begin{equation}\label{eqdevi2}
\nabla_{A^{*}(u)}B(\xi)=(u(A\xi))^{u},
\end{equation}
where $(u(A\xi))^{u}\in T_{u}(L(M))$ is the horizontal lift of $u(A\xi)$.
\end{thm}
\begin{proof}
Let $B=B(\xi)$, $a_{t}=\exp(tA)$, $\widetilde{u}\in L(M)$ and $u_{t}=\widetilde{u}a_{t}$, by definition of covariant derivative (see \cite{KN}, chapter 3, section 1, page 114) we get
\begin{multline}\label{eqmult1}
\nabla_{\dot{u}_{t}}B=\lim_{h\rightarrow 0}\dfrac{1}{h}\left[u_{t}^{t+h}(B(u_{t+h}))-B(u_{t})\right]
\\
=\lim_{h\rightarrow 0}\dfrac{1}{h}\left[u_{t}^{t+h}(u_{t+h}\xi)^{u_{t+h}}-(u_{t}\xi)^{u_{t}}\right]
\\
=\lim_{h\rightarrow 0}\dfrac{1}{h}\left[(u_{t}(a_{h}\xi))^{u_{t}}-(u_{t}\xi)^{u_{t}}\right]
\\
=\lim_{h\rightarrow 0}\dfrac{1}{h}\left[u_{t}((a_{h}-a_{0})\xi)\right]^{u_{t}}
\\
=\left(u_{t}\left(\left[\lim_{h\rightarrow 0}\dfrac{a_{h}-a_{0}}{h}\right]\xi\right)\right)^{u_{t}}=(u_{t}(A\xi))^{u_{t}}.
\end{multline}
The result follows by substituting $u_{t}$ for $u$ given $\dot{u}_{t}=A^{*}(u_{t})$.
\end{proof}

\begin{cor}\label{cordevi1}
Let $(\xi_{j})_{j=1,...,n}$ be a basis of $F$, $B_{j}=B(\xi_{j})$ and $A\in\mathfrak{g}$, then the components of the matrix representation of $A$ with respect to $(\xi_{j})_{j=1,...,n}$ are given by the coefficients $\Gamma^{k}_{j}$ where
\begin{equation}
\nabla_{A^{*}}B_{j}=\Gamma^{k}_{j}B_{k}.
\end{equation}
\end{cor}
\begin{proof}
By Theorem \ref{teodevi}
\[
\nabla_{A^{*}(u)}B_{j}=(u(A\xi_{j}))^{u}=\Gamma^{k}_{j}B_{k}(u)=\Gamma^{k}_{j}(u(\xi_{k}))^{u}=\left(u\left(\Gamma^{k}_{j}\xi_{k}\right)\right)^{u},
\]
where the horizontal lift $(\cdot)^{u}$ is linear given that $\pi_{*}$ is linear. Next we observe that
\[
u(A\xi_{j})=\pi_{*}\left((u(A\xi_{j}))^{u}\right)=\pi_{*}\left(\left(u\left(\Gamma^{k}_{j}\xi_{k}\right)\right)^{u}\right)=u\left(\Gamma^{k}_{j}\xi_{k}\right),
\]
as $u$ defines a isomorphism between $F$ and $T_{\pi(u)}(M)$ we have that
\[
A\xi_{j}=\Gamma^{k}_{j}\xi_{k}.
\]
\end{proof}

\begin{cor}\label{cordevi2}
Given a frame on $L(M)$ uniquely composed by fundamental and horizontal standard vector fields, the coefficients for the induced connection $\nabla$ associated with such fields are all constant. Also, for every basis $(\xi_{j})_{j=1,...,n}$ of $F$ there exists, modulo order, a unique basis $(A_{i})_{i=1,...,n^{2}}$ of $\mathfrak{g}$ such that the induced connection coefficients associated to the frame given by the fields $A^{*}_{i}$ and $B_{j}=B(\xi_{j})$ are such that $n^{2}$ of them are equal to one and the others equal to zero. As the dimension of a fiber is $n+n^{2}$, the total number of coefficients is $(n+n^{2})^{3}$.
\end{cor}
\begin{proof}
By Lemmas \ref{lemdicf} and \ref{lemdich} and Corollary \ref{cordevi1}, it is sufficient to consider the basis $(A_{i})_{i=1,...,n^{2}}$ of $\mathfrak{g}$ given by
\[
A_{i}\xi_{j}=\delta^{k}_{j}\xi_{l},
\]
where $i=n(k-1)+l$ with $k=1,..,n$ y $l=1,...,n$.
\end{proof}

\begin{defin}\label{defdevia}
A frame composed by fields $A^{*}_{i}$ and $B_{j}$ as described in the second part of Corollary \ref{cordevi2} will be called \textit{an adapted frame for the induced connection}.
\end{defin}

\begin{lem}[\cite{KN}, chapter 3, proposition 2.3, page 119]\label{lemconm}
Let $M$ be a manifold with a linear connection. Given a fundamental field $A^{*}$ and a horizontal standard field $B(\xi)$ on $L(M)$, the following holds
\begin{equation}\label{eqconm}
[A^{*},B(\xi)]=B(A\xi).
\end{equation}
\end{lem}

\begin{lem}\label{lemtorch}
Let $M$ be a manifold with linear connection. Given two horizontal standard fields $B(\xi)$ and $B(\zeta)$ on $L(M)$
\begin{equation}\label{eqtorch}
T(B(\xi),B(\zeta))=-[B(\xi),B(\zeta)].
\end{equation}
\end{lem}
\begin{proof}
The result follows from theorem 2.4 from (\cite{KN}, chapter 3, page 120) and lemma \ref{lemdich}.
\end{proof}

\begin{thm}\label{teotori}
For an adapted frame (definition \ref{defdevia}) composed by fields $A_{(\alpha,\beta)}^{*}$ and $B_{j}=B(\xi_{j})$ where the vectors $\xi_{j}$ form a basis of $F$, and the subindices $(\alpha,\beta)$ are organized so that
\[
A_{(\alpha,\beta)}\xi_{k}=\delta_{k}^{\alpha}\xi_{\beta}.
\]
Then the connection coefficients are all zero except for those of the form $\Gamma_{(\alpha,\beta)\alpha}^{\beta}=1$. Let $X=X^{i}B_{i}+X^{(\alpha,\beta)}A_{(\alpha,\beta)}^{*}$ and $Y=Y^{j}B_{j}+Y^{(\gamma,\epsilon)}A_{(\gamma,\epsilon)}^{*}$, the induced torsion is given by
\begin{equation}\label{eqtori}
T(X,Y)=\Sigma_{j}(X^{(i,j)}Y^{(j,k)}-Y^{(i,j)}X^{(j,k)})A_{(i,k)}^{*}-X^{i}Y^{j}[B_{i},B_{j}].
\end{equation}
\end{thm}
\begin{proof}
Applying equation $\nabla_{X}Y=X^{i}\left(e_{i}Y^{k}+\Gamma_{ij}^{k}Y^{j}\right)e_{k}$ for any frame $\{e_{i}\}$, we get
\begin{multline}\label{eqmult2}
\nabla_{X}Y=X^{i}(B_{i}Y^{j})B_{j}+X^{(\alpha,\beta)}(A_{(\alpha,\beta)}^{*}Y^{j})B_{j}+X^{i}(B_{i}Y^{(\gamma,\epsilon)})B_{(\gamma,\epsilon)}
\\
+X^{(\alpha,\beta)}(A_{(\alpha,\beta)}^{*}Y^{(\gamma,\epsilon)})B_{(\gamma,\epsilon)}+\Gamma_{(\alpha,\beta)\alpha}^{\beta}X^{(\alpha,\beta)}Y^{\alpha}B_{\beta}.
\end{multline}
Also, for any pair of fields $Z$ and $W$, and any pair of scalar functions $f$ and $g$
\[
[fZ,gW]=f(Zg)W-g(Wf)Z+fg[Z,W].
\]
So when calculating $\nabla_{X}Y-\nabla_{Y}X$, we can group the first four pairs of terms, adding and subtracting expressions of the form $X^{i}Y^{j}[B_{i},B_{j}]$ to form brackets, obtaining
\begin{multline}\label{eqmult3}
\nabla_{X}Y-\nabla_{Y}X=[X^{i}B_{i},Y^{j}B_{j}]-X^{i}Y^{j}[B_{i},B_{j}]
\\
+[X^{(\alpha,\beta)}A_{(\alpha,\beta)}^{*},Y^{j}B_{j}]-X^{(\alpha,\beta)}Y^{j}[A_{(\alpha,\beta)}^{*},B_{j}]
\\
+[X^{i}B_{i},Y^{(\gamma,\epsilon)}B_{(\gamma,\epsilon)}]-X^{i}Y^{(\gamma,\epsilon)}[B_{i},A_{(\gamma,\epsilon)}^{*}]
\\
+[X^{(\alpha,\beta)}A_{(\alpha,\beta)}^{*},Y^{(\gamma,\epsilon)}B_{(\gamma,\epsilon)}]-X^{(\alpha,\beta)}Y^{(\gamma,\epsilon)}[A_{(\alpha,\beta)}^{*},A_{(\gamma,\epsilon)}^{*}]
\\
+\Sigma_{\alpha}(X^{(\alpha,\beta)}Y^{\alpha}-Y^{(\alpha,\beta)}X^{\alpha})B_{\beta}.
\end{multline}
Applying the bilinearity of the bracket, Lemma \ref{lemconm} and Proposition 4.1 from (\cite{KN}, chapter 1, page 42) we obtain
\begin{multline}\label{eqmult4}
\nabla_{X}Y-\nabla_{Y}X=[X,Y]+\Sigma_{i}(X^{(i,j)}Y^{i}-Y^{(i,j)}X^{i})B_{j}
\\
+(X^{i}Y^{(j,k)}-Y^{i}X^{(j,k)})B(A_{(j,k)}\xi_{i})-X^{(i,j)}Y^{(l,k)}[A_{(i,j)},A_{(l,k)}]^{*}-X^{i}Y^{j}[B_{i},B_{j}].
\end{multline}
Given that $B(A_{(j,k)}\xi_{i})=\delta_{i}^{j}B_{k}$ y $[A_{(i,j)},A_{(l,k)}]^{*}=\delta_{i}^{k}A_{(l,j)}^{*}-\delta_{j}^{l}A_{(i,k)}^{*}$ we obtain
\[
\nabla_{X}Y-\nabla_{Y}X=[X,Y]+\Sigma_{j}(X^{(i,j)}Y^{(j,k)}-Y^{(i,j)}X^{(j,k)})A_{i,k}^{*}-X^{i}Y^{j}[B_{i},B_{j}].
\]
The result follows by applying theorem 5.1 from (\cite{KN}, chapter 3, page 133).
\end{proof}

\begin{thm}\label{teocurvi}
For an adapted frame composed by the fields $A_{(\alpha,\beta)}^{*}$ and $B_{j}$. Let $X=X^{i}B_{i}+X^{(\alpha,\beta)}A_{(\alpha,\beta)}^{*}$, $Y=Y^{j}B_{j}+Y^{(\gamma,\epsilon)}A_{(\gamma,\epsilon)}^{*}$ and $Z=Z^{k}B_{k}+Z^{(\rho,\lambda)}A_{(\rho,\lambda)}^{*}$, the induced curvature is given by
\begin{equation}\label{eqcurvi}
R(X,Y)Z=-X^{i}Y^{j}Z^{k}\nabla_{[B_{i},B_{j}]}B_{k}.
\end{equation}
\end{thm}
\begin{proof}
As a tensor field, curvature depends on $X$, $Y$ and $Z$ as vectors, not as fields, so we may assume that the coefficients of the fields are constant, then
\begin{multline}\label{eqmult5}
[X,Y]=X^{i}Y^{j}[B_{i},B_{j}]+X^{(\alpha,\beta)}Y^{j}[A_{(\alpha,\beta)}^{*},B_{j}]+X^{i}Y^{(\gamma,\epsilon)}[B_{i},A_{(\gamma,\epsilon)}^{*}]+X^{(\alpha,\beta)}Y^{(\gamma,\epsilon)}[A_{(\alpha,\beta)}^{*},A_{(\gamma,\epsilon)}^{*}]
\\
=X^{i}Y^{j}[B_{i},B_{j}]+(X^{(\alpha,\beta))}Y^{i}-X^{i}Y^{(\alpha,\beta)})B(A_{(\alpha,\beta)}\xi_{i})+X^{(\alpha,\beta)}Y^{(\gamma,\epsilon)}[A_{(\alpha,\beta)},A_{(\gamma,\epsilon)}]^{*}
\\
=X^{i}Y^{j}[B_{i},B_{j}]+\Sigma_{i}(X^{(i,j)}Y^{i}-X^{i}Y^{(i,j)})B_{j}+X^{(i,j)}Y^{(k,l)}(\delta^{l}_{i}A_{(k,j)}-\delta^{j}_{k}A_{(i,l)})^{*}
\\
=X^{i}Y^{j}[B_{i},B_{j}]+\Sigma_{i}(X^{(i,j)}Y^{i}-X^{i}Y^{(i,j)})B_{j}+(X^{(k,j)}Y^{(i,k)}-X^{(i,k)}Y^{(k,j)})A_{(i,j)}^{*}.
\end{multline}
As we are assuming constant coefficients, and because the induced covariant derivative of a fundamental field in any direction, as well as the induced covariant derivative of a horizontal standard field in horizontal directions, are null, we have
\begin{multline}\label{eqmult6}
\nabla_{[X,Y]}Z=\nabla_{[X,Y]}Z^{k}B_{k}=X^{i}Y^{j}Z^{k}\nabla_{[B_{i},B_{j}]}B_{k}+(X^{(k,j)}Y^{(i,k)}-X^{(i,k)}Y^{(k,j)})Z^{k}\nabla_{A_{(i,j)}^{*}}B_{k}
\\
=X^{i}Y^{j}Z^{k}\nabla_{[B_{i},B_{j}]}B_{k}+(X^{(k,j)}Y^{(i,k)}-X^{(i,k)}Y^{(k,j)})Z^{k}B(A_{(i,j)}\xi_{k})
\\
=X^{i}Y^{j}Z^{k}\nabla_{[B_{i},B_{j}]}B_{k}+\Sigma_{i}\Sigma_{k}(X^{(k,j)}Y^{(i,k)}-X^{(i,k)}Y^{(k,j)})Z^{i}B_{j}.
\end{multline}
Also, given that $\Gamma_{(i,j)l}^{k}=\delta_{i}^{l}\delta_{j}^{k}$ and all the other induced connection coefficients are zero
\begin{multline}\label{eqmult7}
\nabla_{X}\nabla_{Y}Z=\nabla_{X}(Y^{(i,j)}\Gamma_{(i,j)l}^{k}Z^{l})B_{k}=\nabla_{X}(\Sigma_{i}Y^{(i,j)}Z^{i})B_{j}
\\
=(\Sigma_{i}X^{(\alpha,\beta)}\Gamma_{(\alpha,\beta)\gamma}^{\epsilon}Y^{(i,\gamma)}Z^{i})B_{\epsilon}=\Sigma_{i}\Sigma_{k}X^{(i,j)}Y^{(k,i)}Z^{k}B_{j},
\end{multline}
and then
\[
[\nabla_{X},\nabla_{Y}]Z=\Sigma_{i}\Sigma_{k}(X^{(i,j)}Y^{(k,i)}-X^{(k,i)}Y^{(i,j)})Z^{k}B_{j}.
\]
Because of Theorem 5.1 from (\cite{KN}, chapter 3, page 133) we get
\[
R(X,Y)Z=-X^{i}Y^{j}Z^{k}\nabla_{[B_{i},B_{j}]}B_{k}.
\]
\end{proof}
\begin{rem}\label{obscurvi}
The induced curvature is always horizontal and depends only on the horizontal components of $X$, $Y$ and $Z$, and hence if any of them is vertical, the induced curvature is zero.
\end{rem}

\section{Classification of lifts}

\begin{rem}\label{obsclal1}
In a general sense, we can define in many ways the lift of a linear connection $\widetilde{\nabla}$ on a differential manifold $M$. Given two vectors $X\in TM$, $Z\in T(L(M))$ and a field $Y$ on $M$, let $X^{h}$ and $Y^{h}$ be the horizontal lifts of $X$ and $Y$ on $L(M)$, we can define different kinds of lifts $\nabla$, some with more restrictive conditions

(0) $\pi_{*}\left(\nabla_{X^{h}}Y^{h}\right)=\widetilde{\nabla}_{X}Y$,

(1) $\pi_{*}\left(\nabla_{Z}Y^{h}\right)=\widetilde{\nabla}_{\pi_{*}(Z)}Y$,

(2) $\nabla_{Z}Y^{h}=(\widetilde{\nabla}_{\pi_{*}(Z)}Y)^{h}$,

\noindent
where $(\cdot)^{h}$ represents the horizontal lift of tangent vectors. Clearly each equation is more restrictive than the previous one, given that $Z$ is not necessarily horizontal. Also when comparing on $L(M)$ the vertical components have to coincide. We can have a more restrictive condition out of equation (1) by considering $Y\in\mathfrak{X}(L(M))$ as any projectable field, no necessarily horizontal, meaning

(3) $\pi_{*}\left(\nabla_{Z}Y\right)=\widetilde{\nabla}_{\pi_{*}(Z)}\pi_{*}\left(Y\right)$.
\end{rem}

\begin{defin}\label{defcmel}
Conditions for a lift imposed by equations (0) and (1) on remark $\ref{obsclal1}$, where $Y^{h}$ is strictly horizontal, will be called \textit{minimal condition} and \textit{special condition} respectively.
\end{defin}

\begin{defin}\label{defclal}
We define different \textit{lift classes} identified by the following conditions imposed on a lift $\nabla$ of a linear connection $\widetilde{\nabla}$ over a differentiable manifold $M$ in terms of parallel transport on $L(M)$

(0) $\nabla$ satisfies minimal condition for a lift.

(1) $\nabla$ satisfies special condition for a lift.

(2) $\nabla$ satisfies (1) and also parallel transport preserves horizontality.

(3) $\nabla$ satisfies (1) and also parallel transport preserves verticality.

(4) $\nabla$ satisfies (1) and also $\omega(\nabla_{X}A^{*})=0$ for any vector $X$ and fundamental field $A^{*}$.

(5) $\nabla$ satisfies (1) and also parallel transport preserves both horizontality and verticality.

(6) $\nabla$ satisfies (1) and also parallel transport preserves the connection form.

(7) $\nabla$ satisfies (1) and also parallel transport preserves fundamental fields.

(8) $\nabla$ is the unique induced parallel transport.
\end{defin}

\begin{rem}\label{obsclal2}
Lift classes (0), (1), (2) and (3) are determined by equations (0), (1), (2) and (3) from remark $\ref{obsclal1}$.
\end{rem}

\begin{lem}\label{lemclal}
Lift classes satisfy the following relations

(1) Intersection of classes two and three equal class five.

(2) Intersection of classes two and four equal class six.

(3) Intersection of classes three and four equal class seven.

(4) Intersection of classes two, three and four equal class eight.
\end{lem}

\begin{proof}
All follows from the definitions.
\end{proof}

\begin{rem}\label{obsclal3}
Let us consider additional conditions for a linear connection on $L(M)$, which are not enough to have a lift but do help us better understand the relations among classes

(a) Parallel transport preserves horizontality.

(b) Parallel transport preserves verticality.

(c) For any vector $X$ and fundamental field $A^{*}$, $\omega(\nabla_{X}A^{*})=0$.

\noindent
Clearly, classes (2), (3) and (4) correspond with class one restricted by conditions (a), (b) or (c) respectively.
\end{rem}

\begin{thm}\label{teoclal}
Let the fields $A^{*}_{i}$ and $B_{j}=B(\xi_{j})$ conform a frame composed of fundamental and standard horizontal fields exclusively. Let $X$ be any of such fields, we divide the coefficients of any connection $\nabla$ on $L(M)$ for the previous frame in four sets, determined by correspondence with one of the following expressions

(1) $h(\nabla_{X}B_{j})$,

(2) $v(\nabla_{X}B_{j})$,

(3) $h(\nabla_{X}A^{*}_{i})$,

(4) $v(\nabla_{X}A^{*}_{i})$.

\noindent
The first set is determined according to Lemma $\ref{lemdich}$ and Corollary $\ref{cordevi1}$ if and only if $\nabla$ satisfies the special condition. The other three sets of coefficients correspond respectively with conditions (a), (b) and (c) from remark $\ref{obsclal3}$ so that the condition is satisfied if and only if the coefficients of the corresponding set are all equal to zero.
\end{thm}

\begin{proof}
By definition, special condition determines the canonical projection of the covariant derivatives of horizontal fields, and hence it does determine its horizontal components, so it also determines the coefficients in the first set to be the same as it was for induced transport. For the other 3 sets, observe that conditions (a), (b) and (c) are equivalent to the corresponding expressions being all equal to zero.
\end{proof}

\begin{rem}\label{obsclal4}
Not all set of coefficients on Theorem $\ref{teoclal}$ have the same cardinality, the first one has $n^{2}(n^{2}+n)$ elements, the second and third have $n^{3}(n^{2}+n)$ elements and the last one has $n^{4}(n^{2}+n)$ elements.
\end{rem}

\begin{rem}\label{obsclal5}
We could have much more classes by splitting in two each of the sets in Theorem $\ref{teoclal}$ and each of the conditions in remark $\ref{obsclal3}$ considering separately whether $X$ (or the path) is strictly horizontal or vertical. Along with minimal and special conditions this gives rise to a total of 8 conditions and $2^{7}$ subclasses of lifts.
\end{rem}

\begin{cor}\label{corclal}
As in Theorem $\ref{teoclal}$, the coefficients implied in expressions of the form $h(\nabla_{B_{j}}B_{k})$ are all equal to zero (as in Lemma $\ref{lemdich}$) if and only if $\nabla$ satisfies the minimal condition.
\end{cor}

\begin{rem}
Subclasses in Remark \ref{obsclal5} could be extended, requiring the coefficient in each set to be just constants (not necessarily zero). These conditions are well define, meaning that if we exchange one of this frames for another of the same kind (compose only of fundamental and horizontal standard fields) the coefficients may change their values but will remain to be constants. An even weaker set of conditions would require the coefficients to be constants only when restricted to any fiber.
\end{rem}

\section{Geodesic modeling}

\begin{defin}\label{defvertproy}
For a manifold with connection $M$, we define the \textit{vertical projection} $\nu:T(L(M))\rightarrow\Omega_{1}^{1}(M)$ of $X\in T_{u}(L(M))$ given by
\begin{equation}\label{eqvertproy}
\nu(X)=u\omega(X)u^{-1}.
\end{equation}
We also define the \textit{vertical lift} of $F\in\Omega_{1}^{1}(M)$ as the unique vertical vector in $T_{u}(L(M))$ that vertically projects on $F$.
\end{defin}

\begin{defin}\label{deflevaj}
Given $n^{2}$ linearly independent (point by point) fields $F_{i}\in\Omega_{1}^{1}(M)$, we define the \textit{adjust lift} of a connection on a manifold $M$ as the unique lift of which parallel transport preserves horizontality, verticality, the canonical form and the vertical lifts of the $F_{i}$.
\end{defin}

\begin{rem}
The tuple $F_{i}$ may only exist locally, although, if we fix a particular field $F=F_{1}$ everywhere different from zero, using partitions of the unit it is possible to combine local adjust liftings in order to construct a global lifting that respects the parallelism of $F$ everywhere, and keeps preserver horizontality, verticality and the canonical form.
\end{rem}

\subsection{Proof of Theorem 1}
We are ready to give a proof of Theorem \ref{teolevaj}.
\begin{proof}
Let $F_{1}=F$ and let us choose the other $n^{2}-1$ fields $F_{i}$, so that $F_{1}$,$F_{2}$, ... , $F_{n^{2}}$ are linearly independent (point by point). Let $\nabla$ be the adjust lift of $\widetilde{\nabla}$ determinated by the fields $F_{i}$. Let $\gamma$ be the geodesic for $\nabla$ such that $\theta\left(\dot{\gamma}\right)=\xi$ and $\nu\left(\dot{\gamma}\right)=F$. Let $\tau=\pi(\gamma)$ and $\phi=\tau^{h}$ its horizontal lift from $\gamma(0)=\phi(0)$. We shall see that in fact $\widetilde{\nabla}_{\dot{\tau}}\dot{\tau}=F\dot{\tau}$. 
By definition
\begin{equation}\label{eqdemlevaj}
\widetilde{\nabla}_{\dot{\tau}}\dot{\tau}=\lim_{t\rightarrow0}\dfrac{\tau^{-1}(\dot{\tau}_{t})-\dot{\tau}_{0}}{t},
\end{equation}
where $\tau^{-1}$ represents the parallel transport along $\tau$, from $T_{\tau_{t}}(M)$ to $T_{\tau_{0}}(M)$. Let $\gamma_{t}=\phi_{t}g_{t}$ with $g_{t}\in G$, we see that $\phi_{t}\dot{g}_{t}=v\left(\dot{\gamma}_{t}\right)$ which in turn implies
\[
\nu(\phi_{t}\dot{g}_{t})=\gamma_{t}\omega(\phi_{t}\dot{g}_{t})(\gamma_{t})^{-1}=F(\tau_{t}).
\]
On the other hand, $\theta\left(\dot{\phi}_{t}\right)=g_{t}\xi$ and hence
\[
\dot{\tau}_{t}=\pi_{*}\left(\dot{\gamma_{t}}\right)=\gamma_{t}\xi=\phi_{t}(g_{t}\xi)=\pi_{*}\left(\dot{\phi_{t}}\right).
\]
As parallel transport on $L(M)$ preserves $\theta$ and since $\phi$ is horizontal, parallel transport along $\phi$ also commutes with the canonical projection and parallel transport on $M$, which gives
\begin{equation}\label{eqdemlevaju}
\widetilde{\nabla}_{\dot{\tau}}\dot{\tau}=\lim_{t\rightarrow0}\dfrac{\phi_{0}(g_{t}\xi-\xi)}{t}=F\dot{\tau},
\end{equation}
given
\[
F\dot{\tau}=\gamma_{t}\omega(\phi_{t}\dot{g}_{t})(\gamma_{t})^{-1}\gamma_{t}\xi=\gamma_{t}\omega(\phi_{t}\dot{g}_{t})\xi=\gamma_{t}\dot{g}_{t}\xi,
\]
and also $\gamma_{0}=\phi_{0}$.
\end{proof}

\section{Examples and applications}\label{c5secsgea}

Given a geodesically complete manifold $M$ with connection $\nabla$ let us consider the family of equations (with $k\in\mathbb{R}-\{0\}$)

\begin{equation}\label{eqfami}
\nabla_{\dot{\gamma}}\dot{\gamma}=kF\dot{\gamma},
\end{equation}
where $F\in\Omega_{1}^{1}(M)$ is fixed and everywhere different form zero. We obtain that if a curve is solution for one of these equations, it is a solution for all the others when varying the time scale. By replacing and using the chain rule we obtain the next lemma.

\begin{lem}\label{lemk}
The curve $\gamma(t)=\sigma(kt)$ is solution of \eqref{eqfami} if and only if $\nabla_{\dot{\sigma}}\dot{\sigma}=F\dot{\sigma}$.
\end{lem}

\begin{rem}\label{obsss}
Let $a,b\in M$ and $X_{a}\in T_{a}(M)$ such that the geodesic in the direction of $X_{a}$ is locally the only one that connects $a$ and $b$. This can be expressed in terms of the geodesic flow function $\Phi:\mathbb{R}\times TM\rightarrow M$ as $\Phi_{t}(X_{a})=b$, where $t=1$ is the time during which we must follow geodesic flow to get from one point to the other. If we now consider the geometry of $L(M)$, Lemma \ref{lemk} can also be seen as a consequence of the following property of any geodesic flow $\Phi_{t}(kX)=\Phi_{kt}(X)$.
\end{rem}

\begin{defin}
Let $\Psi_{k}:TM\rightarrow M$ be the flow of \eqref{eqfami} for $t=1$, we define for $k=0$
\begin{equation}\label{eqfluy}
Y_{1}=(d\Phi_{1}(X_{a}))^{-1}\left(\dfrac{d}{dk}\left(\Psi_{k}(X_{a})\right)\right)\in T_{X_{a}}(T_{a}(M))\cong T_{a}(M).
\end{equation}
The invertibility of $d\Phi_{1}$ is locally guaranteed by geodesic completeness of $M$, the existence of the derivative of $\Psi$ respect to $k$ is guaranteed by the fact that the eigenvalues of $F$ are locally bounded.
\end{defin}

\begin{thm}\label{teofront}
Let $\gamma_{k}$ be the solution of \eqref{eqfami} such that $\gamma_{k}(0)=a$ and $\gamma_{k}(1)=b$ are fixed. Given $\dot{\gamma}_{0}(0)=X_{a}$, this boundary value problem has a solution if there is a function $Y:\mathbb{R}\rightarrow TM$ such that $\dot{\gamma}_{k}(0)=X_{a}+Y(k)$. In this case we observe that the linear approximation of $Y$ is $Y'(0)=-Y_{1}$.
\end{thm}
\begin{proof}
Given that for $k=0$ we have $\Phi_{1}=\Psi_{k}$ and $(d/dk)\Psi_{k}(X_{a})=d\Phi_{X_{a}}(Y_{1})$, we see (for $k=0$) $(d/dk)\Psi_{k}(X_{a}-kY_{1})=0$, so $(d/dk)\Psi_{k}(X_{a}-kY_{1})=kZ(k)$, for some function $Z$. Because including the linear term $-kY_{1}$ allows us to eliminate the constant term on the derivative of $\Psi_{k}$ with respect to $k$ we see that in fact it is the linear approximation of the solution.
\end{proof}

\begin{rem}
The definition of $Y_{1}$ can be generalized by considering higher order derivatives with respect to $k$ along with the convention $Y_{0}=X_{a}$ in order to construct a power series. Finding out if in this way we can obtain a better solution goes beyond the scope of this investigation.
\end{rem}

\begin{exam}
Let us consider the family of equations \eqref{eqfami} on $M=SO(S^{2})$ where
\[
F=\left(
\begin{array}{cc}
0&-1\\
1&0
\end{array}
\right).
\]
By symmetry, the solutions are circumferences. To compute the radius as function of the speed, without loss of generality we assume that there is a solution
\[
\gamma(t)=(r\cos(\alpha t),r\sin(\alpha t),z),
\]
where $r^{2}+z^{2}=1$ and the speed is $v=\alpha r$. we see that
\[
\nabla_{\dot{\gamma}}\dot{\gamma}=\ddot{\gamma}(0)-Proy_{\gamma(0)}\ddot{\gamma}(0)=rz\alpha^{2}(-z,0,r).
\]
Given $\nabla_{\dot{\gamma}}\dot{\gamma}=kF\dot{\gamma}$ we obtain that $zv=rk$ and hence
\[
r=\dfrac{v}{\sqrt{v^{2}+k^{2}}}.
\]
It is evident that $r$ tends to 1 on both cases, when $k$ tends to zero or when $v$ tends to infinity, and that if $v$ tends to zero $r$ also does. At the same time, although Theorem \ref{teofront} is about existence of solution for \eqref{eqfami} which connects to given points $a,b\in M$, this example shows us that this is only locally true, because a circumference which connects two antipodal points must have radius one, but for $k\neq0$ such trajectory requires infinite speed.
\end{exam}

\begin{exam}
Let us consider the family of equations \eqref{eqfami} on $M=\mathbb{R}^{2}$ where
\[
F=\left(
\begin{array}{cc}
0&-1\\
1&0
\end{array}
\right).
\]
Solutions are circumferences with angular velocity $k$. Let us fix as boundary data $(1,y)$ and $(-1,y)$, where the value of $y$ is not relevant. For the time to go from one point to the other to be 1, the trajectory must be an arc of $k$ radians, so the radius of the circumferences should be $r=\sec(k/2)$. Taking initial and final time as $t=\pi/2k\mp1/2$ respectively along with the parametrization $\gamma_{k}(t)=r(cos(kt),sen(kt))$, we observe that $\dot{\gamma}_{k}(t_{a})=(-k\cot(k/2),k)$.
\end{exam}

\end{document}